\documentclass[a4paper,reqno,final,11pt]{amsart}
\usepackage{amssymb}
\usepackage{amsmath}
\usepackage{amsfonts}
\usepackage{hyperref}
\usepackage{graphicx}

\newtheorem{theorem}{Theorem}
\theoremstyle{plain}
\newtheorem{corollary}{Corollary}

\newtheorem{remark}{Remark}
\numberwithin{equation}{section}

\email{amsengouga@gmail.com,abdelmouhcene.sengouga@univ-msila.dz}

\begin{document}
\title[Free vibrations of axially moving strings]{Free vibrations of axially
moving strings: Energy estimates and boundary observability}
\author[S. Ghenimi]{Seyf Eddine Ghenimi}
\author[A. Sengouga]{Abdelmouhcene Sengouga}
\address[Seyf Eddine Ghenimi, Abdelmouhcene Sengouga]{ Laboratory of Functional Analysis and Geometry of Spaces\\
Department of mathematics\\
Faculty of Mathematics and Computer Sciences\\
University of M'sila\\
28000 M'sila, Algeria.}
\date{\today}
\maketitle

\begin{abstract}
We study the small vibrations of axially moving strings described by a wave
equation in an interval with two endpoints moving in the same direction with
a constant speed. The solution is expressed by a series formula where the
coefficients are explicitly computed in function of the initial data. We
also define an energy expression for the solution that is conserved in time.
Then, we establish boundary observability inequalities with explicit
constants.
\end{abstract}


\section{Introduction}

The present work deals with small transverse vibrations of an infinite
string moving axially with a constant speed. Two fixed supports, distanced
by $L$ as represented in Figure \ref{fig0}, prevent transversal
displacements of the string at the supporting points while the axial motion
remains unaffected. 
\begin{figure}[tbph]
\centering
\includegraphics[width=0.61\textwidth]{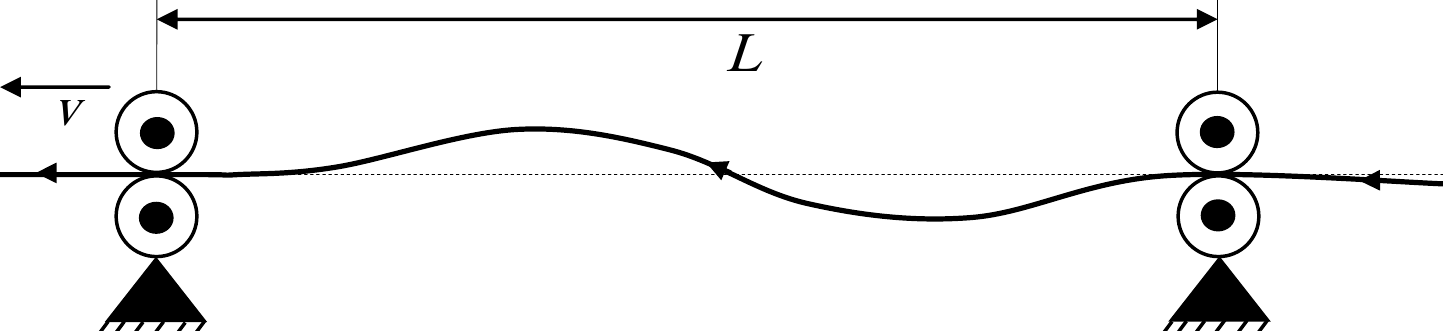}
\caption{A string travelling to the left with a speed $v$.}
\label{fig0}
\end{figure}

We introduce a coordinate system $(x,t)$, attached to the travelling string,
where $x$ coincides with the rest state axis of the string and $t$ denotes
the time. We denote transverse displacement of the string by $\phi (x,t)$
and we choose the position of the left support to coincide with $x=0$.
Assuming that the string travels to the left with a scalar speed $v$, the
positions of the left and right supports are $x=vt$ and $x=L+vt$ for $t\geq
0 $, respectively. If we assume that the string travels to the right then it
suffices to change $v$ by $-v$ in the remainder of this paper.

For $T>0,$ we denote the interval 
\begin{equation*}
\mathbf{I}_{t}:=\left( vt,L+vt\right) ,\text{ \ for\ }t\in \left( 0,T\right)
.
\end{equation*}%
A simplified model describing the free small transverse vibrations of this
string is the following wave equation 
\begin{equation}
\left\{ 
\begin{array}{ll}
\phi _{tt}-\phi _{xx}=0,\ \smallskip & \text{for }x\in \mathbf{I}_{t}\text{
and }t\in \left( 0,T\right) , \\ 
\phi \left( vt,t\right) =\phi \left( L+vt,t\right) =0,\smallskip & \text{for 
}t\in \left( 0,T\right) , \\ 
\phi (x,0)=\phi ^{0}\left( x\right) \text{, \ \ }\phi _{t}\left( x,0\right)
=\phi ^{1}\left( x\right) ,\text{ \ } & \text{for }x\in \mathbf{I}_{0},%
\end{array}%
\right.  \tag{WP}  \label{wave}
\end{equation}%
where the subscripts $t$ and $x$ stand for the derivatives in time and space
variables respectively, $\phi ^{0}$ is the initial shape of the string and $%
\phi ^{1}$ is its initial transverse speed. We assume that the speed $v$ is
strictly less then the speed of propagation of the wave (here normalized to $%
c=1$), i.e. 
\begin{equation}
0<v<1.  \label{tlike}
\end{equation}%
If $v\geq 1,$ then the problem is ill-posed, see for instance \cite{RaCa1995}%
.

The wave equation formulated above is a simple model to represent several
mechanical systems such as plastic films, magnetic tapes, elevator cables,
textile and fibre winding, see for example \cite{Chen2005,BBJT2020,HoPh2019}%
. This model can be dated back to Skutch \cite{Skut1897}, its simplicity is
only apparent and we should mention that the method of separation of
variables cannot be applied to this problem. Miranker's work \cite{Mira1960}
is one of the early influencing papers on the topic of axially moving media.
He proposed two approaches to solve Problem (\ref{wave}). The first one is
to "freeze" the space interval by formulating the problem in the interval $%
\left( 0,L\right) $. Thus, introducing the variables $\eta =x-vt$ and $\tau
=t$, the first equation in (\ref{wave}) becomes%
\begin{equation}
\phi _{\tau \tau }-2v\phi _{\eta \tau }-\left( 1-v^{2}\right) \phi _{\eta
\eta }=0,\text{ \ for }\eta \in \left( 0,L\right) ,\text{ }\tau >0.
\label{freez}
\end{equation}%
The obtained problem is more familiar and the vast majority of the
literature on travelling strings follows this approach. Some important
results in this direction are given by Wickert and Mote \cite{WiMo1990}
where the authors write (\ref{freez}) as a first-order differential equation
with matrix differential operators (a state space formulation) and obtained
a closed form representation of the solution for arbitrary initial
conditions. There is also other methods to solve (\ref{freez}), for instance
a solution by the Laplace transform method is proposed in \cite{vaPo2005}.
The solution can also be constructed using the characteristic method, see
for instance \cite{RaCa1995,CLFL2017}.

The approach of Miranker \cite{Mira1960} is to solve (\ref{wave}), i.e. keep
the space interval depending on time. He obtained a closed form of the
solution by a series formulas (See page 39 in \cite{Mira1960}). After few
rearrangements, his formulas can be rewritten as%
\begin{equation}
\phi (x,t)=\sum_{n\in \mathbb{\mathbb{Z} }^{\ast }}c_{n}\left( e^{n\pi
i\left( 1-v\right) \left( t+x\right) /L}-e^{n\pi i\left( 1+v\right) \left(
t-x\right) /L}\right) ,\text{ for }x\in \mathbf{I}_{t}\text{ and }t\in
\left( 0,T\right) .  \label{exact0}
\end{equation}%
Despite the utility of such a formula for numerical and asymptotic
approaches, it remained underexploited in the literature related to axially
moving strings.

Since Miranker was not explicit on how to compute the coefficients $c_{n}$,
we give in the paper at hands a method to compute each $c_{n}$ in function
of the initial data $\phi ^{0}\ $and $\phi ^{1},$ see Theorem \ref{thexist1}
in the next section. The idea is inspired from \cite{Seng2020} where the
second author obtained the exact solution of strings with two linearly
moving endpoints at different speeds. Similar techniques were used in \cite%
{Bala1961,Seng2018} for a string with one moving endpoint. Each problem in 
\cite{Seng2020,Bala1961,Seng2018}, is set in an interval expanding with time
(in the inclusion sense) and the solution is presented by a series
containing a type of functions different from those in (\ref{exact0}). Thus,
the results of \cite{Seng2020,Seng2018} in particular do not apply to the
present problem (\ref{wave}).

In this work, we show that the series formulas (\ref{exact0}) can be
manipulated to establish the following results:

\begin{itemize}
\item \emph{A conserved quantity.} The functional\footnote{%
Here and in the sequel, the subscript $v$ is used to emphasize the
dependence on the speed $v$.}%
\begin{equation}
\mathcal{E}_{v}\left( t\right) =\frac{1}{2}\int_{vt}^{L+vt}\left( \phi
_{t}+v\phi _{x}\right) ^{2}+\left( 1-v^{2}\right) \phi _{x}^{2}dx,\ \ \ 
\text{for }t\geq 0,  \label{E}
\end{equation}%
depending on $L,t,v$ and the solution of $\left( \ref{wave}\right) ,$ is
conserved in time. We give two different proofs for this fact, see Theorem %
\ref{th1}. Note that $\phi _{t}+v\phi _{x}=\frac{d}{dt}\left( \phi \left(
x+vt,t\right) \right) $ is the total (called also the material) derivative.
Under the assumption (\ref{tlike}), this functional is positive-definite and
we will call it the "energy" of the solution $\phi $. Although there are
many expressions of energy for axially moving strings, see for instance \cite%
{RRWM1998,WiMo1989}, we could not find the definition (\ref{E}) in the
literature.
\end{itemize}

\begin{itemize}
\item \emph{Exact boundary observability.}

\begin{itemize}
\item The wave equation (\ref{wave}) is exactly observable at any endpoint $%
x=x_{b}+vt,$ where $x_{b}=0$ or $x_{b}=L$. Due to the finite speed of
propagation, the time of observability is expected to be positive and
depends on the initial length $L$ and the speed $v$. We show that this time
is exactly 
\begin{equation*}
T_{v}:=2L/(1-v^{2}),
\end{equation*}
see Theorems \ref{thobs1}.

\item If we observe both endpoints, i.e. for $x =vt$ and $x =L+vt,$ the time
of observability is reduced to 
\begin{equation*}
\tilde{T}_{v}:=L/(1-v),
\end{equation*}
see Theorem \ref{thobs2}.
\end{itemize}
\end{itemize}

Although the problem considered here is linear and extensively studied, the
application of Fourier series method to establish the above stated results
is new to the best of our knowledge. Let us also note that letting $%
v\rightarrow 0$ in the above results, we recover some known facts for the
wave equation in non-travelling intervals \cite{Lion1988,KoLo2005}. In
particular, $\mathcal{E}_{0}\left( t\right) =\frac{1}{2}\int_{0}^{L}\phi
_{t}^{2}+\phi _{x}^{2}dx\ $is known to be conserved and we get $T_{0}=2L,%
\tilde{T}_{0}=L$ as sharp values for boundary observability time.

After the present introduction, we derive an expression for the coefficients
of the series formula (\ref{exact0}). In section 3, we show that the energy $%
\mathcal{E}_{v}$ is conserved in time. The boundary observability results at
one endpoint and at both endpoints are addressed in the last section.

\section{Computing the coefficients of the series}

To simplify some formulas, we introduce the notation%
\begin{equation*}
\gamma _{v}:=\frac{1+v}{1-v},\text{ \ \ \ \ }L_{1}:=\frac{1-v}{1+v}L\ \ 
\text{and \ }L_{2}:=\frac{2}{1-v}L
\end{equation*}%
since these constants will appear frequently in the sequel. Note that 
\begin{equation*}
1<\gamma _{v}<+\infty \text{ \ \ and \ \ }0<L_{1}<L<L_{2}/2,\text{ \ for }%
0<v<1.
\end{equation*}%
For every initial data%
\begin{equation}
\phi ^{0}\in H_{0}^{1}\left( \mathbf{I}_{0}\right) ,\text{ \ }\phi ^{1}\in
L^{2}\left( \mathbf{I}_{0}\right) ,  \label{ic}
\end{equation}%
we already know that if (\ref{tlike}) holds the solution of Problem (\ref%
{wave})\emph{\ }exists and satisfies%
\begin{equation}
\phi \in C\left( [0,T];H_{0}^{1}\left( \mathbf{I}_{t}\right) \right) \text{\
\ \ and \ \ }\phi _{t}\in C\left( [0,T];L^{2}\left( \mathbf{I}_{t}\right)
\right) ,  \label{solreg}
\end{equation}%
see for instance \cite{DaZo1990,BaCh1981}. Moreover, an easy computation
shows that the solution $\phi $ given by (\ref{exact0}) satisfies the
periodicity relation 
\begin{equation}
\phi (x+vT_{v},t+T_{v})=\phi (x,t),  \label{T-period}
\end{equation}%
i.e., after a time $T_{v}=2L/\left( 1-v^{2}\right) $ the string travels a
distance $vT_{v}$ and return to its original form at time $t$.

\subsection{Coefficients expressions}

\begin{theorem}
\label{thexist1}Under the assumptions (\ref{tlike}) and (\ref{ic}), the
solution of\ Problem (\ref{wave}) is\emph{\ }given by the series (\ref%
{exact0}) where the coefficients $c_{n}\in \mathbb{C}$ are given by any of
the two following formulas 
\begin{align}
c_{n}& =\frac{1}{4n\pi i}\int_{0}^{L_{2}}\left( \tilde{\phi}_{x}^{0}+\tilde{%
\phi}^{1}\right) e^{-n\pi i\left( 1-v\right) x/L}dx,  \label{cn+} \\
& =\frac{1}{4n\pi i}\int_{-L_{1}}^{L}\left( \tilde{\phi}_{x}^{0}-\tilde{\phi}%
^{1}\right) e^{n\pi i\left( 1+v\right) x/L}dx\text{, \ \ for }n\in \mathbb{\ 
\mathbb{Z}}^{\ast },  \label{cn-}
\end{align}%
where\emph{\ }$\tilde{\phi}_{x}^{0}$ and $\tilde{\phi}^{1}$ are extensions
of the initial data $\phi ^{0}$and\ $\phi ^{1}$ on the interval $\left(
-L_{1},L_{2}\right) $ given below by (\ref{phi0x+}) and (\ref{phi1+})\emph{\ 
}respectively.
\end{theorem}

Before proceeding to the proof, let us describe how to extend the function $%
\phi ,\ $defined only on $\mathbf{I}_{t}=\left( vt,L+vt\right) ,$ to the
intervals $\left( -L_{1}+vt,vt\right) $ and $\left( L+vt,L_{2}+vt\right) $.
On one hand, we set%
\begin{equation}
\tilde{\phi}(x,t)=\left\{ 
\begin{array}{ll}
-\phi \left( \gamma _{v}\left( vt-x\right) +vt,t\right) , & \text{if }x\in
\left( -L_{1}+vt,vt\right) ,\smallskip \\ 
\phi \left( x,t\right) , & \text{if }x\in \left( vt,L+vt\right) ,\smallskip
\\ 
-\phi \left( \frac{1}{\gamma _{v}}\left( vt-x\right) +\frac{2L}{1+v}%
+vt,t\right) ,\text{ \ \ } & \text{if }x\in \left( L+vt,L_{2}+vt\right) .%
\end{array}%
\right.  \label{phi0+}
\end{equation}%
The obtained function is well defined since the first variable of $\phi $
remains in the interval $\left( vt,L+vt\right) $. In particular, $\tilde{\phi%
}(vt,t)=\tilde{\phi}(L+vt,t)=0,\ $hence the homogeneous boundary conditions
at $x=vt$ and $x=L+vt$ remain satisfied, for every $t\geq 0$.

\begin{figure}[tbph]
\centering 
\includegraphics[width=0.71\textwidth]{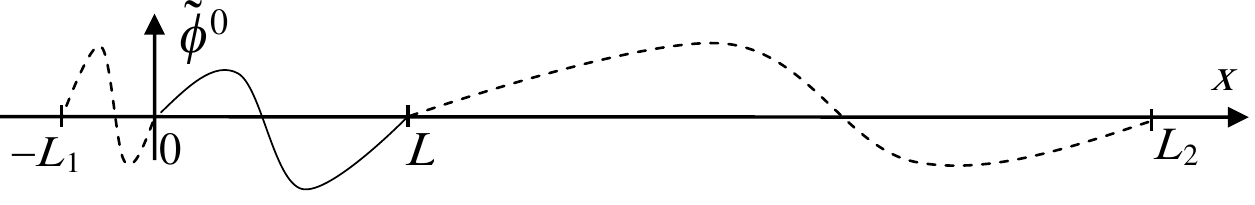}
\caption{Example of the extension of an initial data $\protect\phi ^{0}$.}
\label{fig1a}
\end{figure}

\begin{remark}
If $v=0$, then $L_{1}=L$ and $L_{2}=2L$. In this case, the functions $\tilde{%
\phi}$ and $\tilde{\phi}_{t}$ are odd on the intervals $\left( -L,L\right) $
and $\left( 0,2L\right) $ with respect to the middle of each interval. The
extension $\tilde{\phi}_{x}$ is an even function on these intervals.
\end{remark}

Taking the derivative of (\ref{phi0+}) with respect to $x$, we obtain$,$%
\begin{equation}
\tilde{\phi}_{x}(x,t)=\left\{ 
\begin{array}{ll}
\gamma _{v}\phi _{x}\left( \gamma _{v}\left( vt-x\right) +vt,t\right) , & 
\text{if }x\in \left( -L_{1}+vt,vt\right) ,\smallskip \\ 
\phi _{x}\left( x,t\right) , & \text{if }x\in \left( vt,L+vt\right)
,\smallskip \\ 
\frac{1}{\gamma _{v}}\phi _{x}\left( \frac{1}{\gamma _{v}}\left( vt-x\right)
+\frac{2L}{1+v}+vt,t\right) ,\text{ \ \ \ \ } & \text{if }x\in \left(
L+vt,L_{2}+vt\right) .%
\end{array}%
\right.  \label{phi0x+}
\end{equation}%
On the other hand, $\tilde{\phi}_{t}(x,t)$ is extended as follows%
\begin{equation}
\tilde{\phi}_{t}(x,t)=\left\{ 
\begin{array}{ll}
-\gamma _{v}\phi _{t}\left( \gamma _{v}\left( vt-x\right) +vt,t\right) , & 
\text{if }x\in \left( -L_{1}+vt,vt\right) ,\smallskip \\ 
\phi _{t}\left( x,t\right) , & \text{if }x\in \left( vt,L+vt\right)
,\smallskip \\ 
\frac{-1}{\gamma _{v}}\phi _{t}\left( \frac{1}{\gamma _{v}}\left(
vt-x\right) +\frac{2L}{1+v}+vt,t\right) ,\text{ \ \ \ \ } & \text{if }x\in
\left( L+vt,L_{2}+vt\right) .%
\end{array}%
\right.  \label{phi1+}
\end{equation}

\begin{remark}
In Figure \ref{fig3}, let $(x_{1},t_{1})$ be the intersection of the two
characteristic starting from the initial endpoints $x=0$ and $x=L$, after
one reflection on the boundaries. We can check that, the two backward
characteristic lines from $(x_{1},t_{1})$ intersect the $x-$axis precisely
at $x=-L_{1}$ and $x=L_{2}$.
\end{remark}

\begin{figure}[bth]
\centering 
\includegraphics[width=0.67\textwidth,height=51mm]{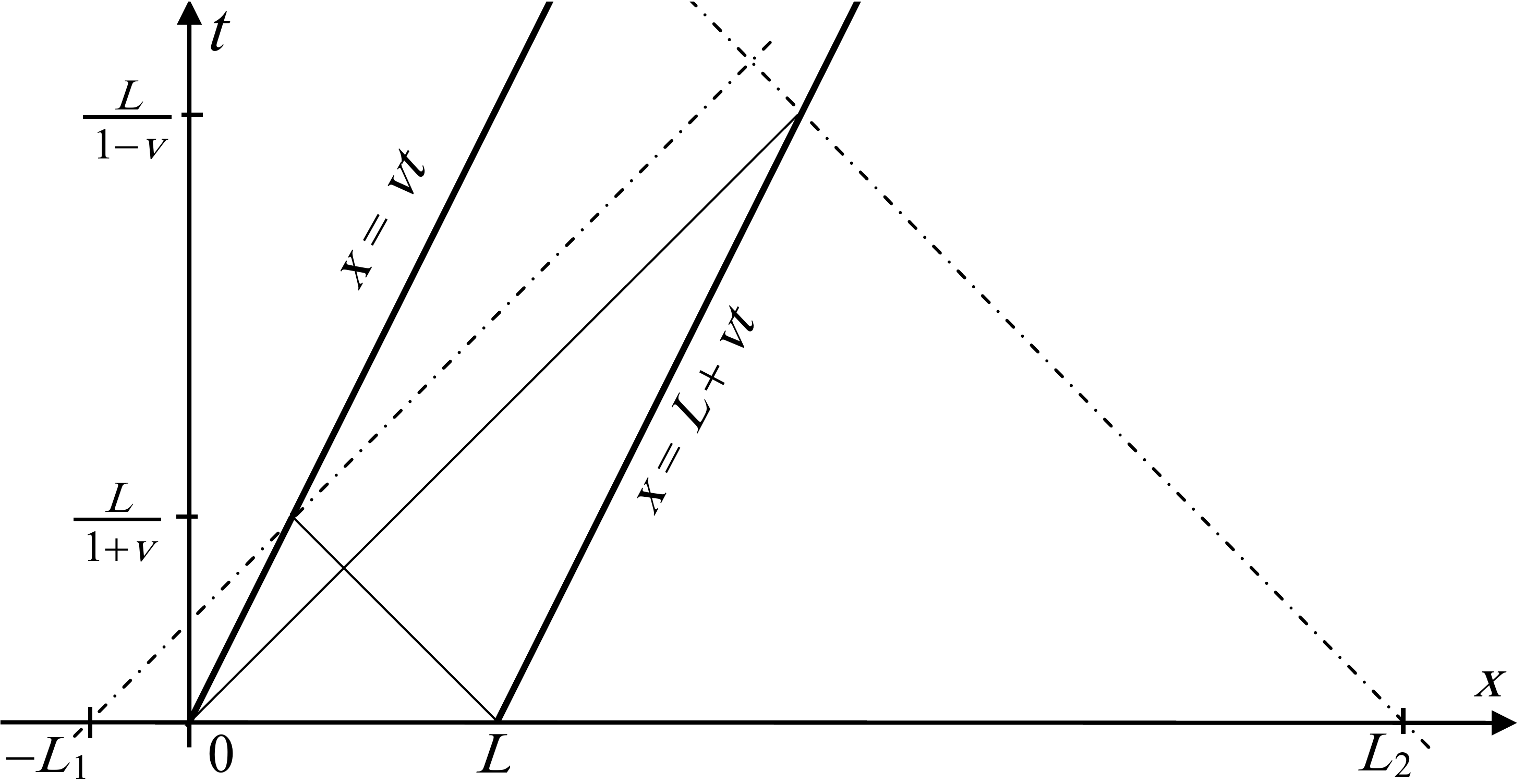}
\caption{Relation between $L_1,L_2$ and some characteristics of wave
propagation.}
\label{fig3}
\end{figure}

Now we are ready to show the coefficients formulas.

\begin{proof}[Proof of Theorem \protect\ref{thexist1}]
Thanks to (\ref{solreg}), we can derive term by term the series (\ref{exact0}%
), it comes that%
\begin{align}
\phi _{x}(x,t)& =\frac{\pi i}{L}\sum_{n\in \mathbb{\mathbb{Z}}^{\ast
}}nc_{n}\left( \left( 1-v\right) e^{n\pi i\left( 1-v\right) \left(
t+x\right) /L}+\left( 1+v\right) e^{n\pi i\left( 1+v\right) \left(
t-x\right) /L}\right) ,\smallskip  \label{ph x} \\
\phi _{t}(x,t)& =\frac{\pi i}{L}\sum_{n\in \mathbb{\mathbb{Z}}^{\ast
}}nc_{n}\left( \left( 1-v\right) e^{n\pi i\left( 1-v\right) \left(
t+x\right) /L}-\left( 1+v\right) e^{n\pi i\left( 1+v\right) \left(
t-x\right) /L}\right) ,  \label{ph t}
\end{align}%
where $t\geq 0$, $x\in \left( vt,L+vt\right) $. Combining this, with (\ref%
{phi0x+}) and (\ref{phi1+}), the extensions $\tilde{\phi}_{x}$ and $\tilde{%
\phi}_{t}$ on the interval $\left( vt,L_{2}+vt\right) $ are given by%
\begin{equation}
\tilde{\phi}_{x}(x,t)=\left\{ 
\begin{array}{ll}
\displaystyle\frac{\pi i}{L}\sum_{n\in \mathbb{\mathbb{Z}}^{\ast
}}nc_{n}\left( \left( 1-v\right) e^{n\pi i\left( 1-v\right) \left(
t+x\right) /L}+\left( 1+v\right) e^{n\pi i\left( 1+v\right) \left(
t-x\right) /L}\right) , & \text{ if\ }x\in \left( vt,L+vt\right) \smallskip ,
\\ 
\displaystyle\frac{\pi i}{\gamma _{v}L}\sum_{n\in \mathbb{\mathbb{Z}}^{\ast
}}nc_{n}\left\{ \left( 1-v\right) e^{\frac{n\pi i\left( 1-v\right) }{L}%
\left( \left( 1+v\right) t+\frac{vt-x}{\gamma _{v}}+\frac{2L}{1+v}\right)
}\right. , &  \\ 
\multicolumn{1}{r}{\left. +\left( 1+v\right) e^{\frac{n\pi i\left(
1+v\right) }{L}\left( \left( 1-v\right) t-\frac{vt-x}{\gamma _{v}}-\frac{2L}{%
1+v}\right) }\right\}} & \multicolumn{1}{r}{\text{ if\ }x\in \left(
L+vt,L_{2}+vt\right) ,}%
\end{array}%
\right.  \label{phx}
\end{equation}

\begin{equation}
\tilde{\phi}_{t}(x,t)=\left\{ 
\begin{array}{ll}
\displaystyle\frac{\pi i}{L}\sum_{n\in \mathbb{\mathbb{Z}}^{\ast
}}nc_{n}\left( \left( 1-v\right) e^{n\pi i\left( 1-v\right) \left(
t+x\right) /L}-\left( 1+v\right) e^{n\pi i\left( 1+v\right) \left(
t-x\right) /L}\right) , & \text{ if }x\in \left( vt,L+vt\right) \smallskip ,
\\ 
\displaystyle\frac{-\pi i}{\gamma _{v}L}\sum_{n\in \mathbb{\mathbb{Z}}^{\ast
}}nc_{n}\left\{ \left( 1-v\right) e^{\frac{n\pi i\left( 1-v\right) }{L}%
\left( \left( 1+v\right) t+\frac{vt-x}{\gamma _{v}}+\frac{2L}{1+v}\right)
}\right. , &  \\ 
\multicolumn{1}{r}{\left. -\left( 1+v\right) e^{\frac{n\pi i\left(
1+v\right) }{L}\left( \left( 1-v\right) t-\frac{vt-x}{\gamma _{v}}-\frac{2L}{%
1+v}\right) }\right\}} & \multicolumn{1}{r}{\text{\ if }x\in \left(
L+vt,L_{2}+vt\right) ,}%
\end{array}%
\right.  \label{pht}
\end{equation}

Taking the sum of (\ref{phx}) and (\ref{pht}) on the interval $\left(
vt,L_{2}+vt\right) $, we get%
\begin{equation*}
\tilde{\phi}_{x}+\tilde{\phi}_{t}=\left\{ 
\begin{array}{ll}
\displaystyle\frac{2\pi i}{L}\left( 1-v\right) \sum_{n\in \mathbb{\mathbb{Z}}%
^{\ast }}nc_{n}e^{n\pi i\left( 1-v\right) \left( t+x\right) /L}, & x\in
\left( vt,L+vt\right) ,\medskip \\ 
\displaystyle\frac{2\pi i}{\gamma _{v}L}\left( 1+v\right) \sum_{n\in \mathbb{%
\mathbb{Z}}^{\ast }}nc_{n}e^{\frac{n\pi i\left( 1+v\right) }{L}\left( \left(
1-v\right) t-\frac{vt-x}{\gamma _{v}}-\frac{2L}{1+v}\right) }, & x\in \left(
L+vt,L_{2}+vt\right) .%
\end{array}%
\right.
\end{equation*}%
Since $e^{\frac{n\pi i\left( 1+v\right) }{L}\left( \left( 1-v\right) t-\frac{%
vt-x}{\gamma _{v}}-\frac{2L}{1+v}\right) }=e^{n\pi i\left( 1-v\right) \left(
t+x\right) /L}$, we get the same expression on the two sub-intervals, i.e. 
\begin{equation}
\tilde{\phi}_{x}+\tilde{\phi}_{t}=\frac{2\pi i}{L}\left( 1-v\right)
\sum_{n\in \mathbb{\mathbb{Z}}^{\ast }}nc_{n}e^{n\pi i\left( 1-v\right)
\left( t+x\right) /L},\text{ \ \ for}\ \ x\in \left( vt,L_{2}+vt\right) .
\label{29}
\end{equation}%
Taking into account that $\left\{ \sqrt{\frac{1-v}{2L}}e^{n\pi i\left(
1-v\right) \left( t+x\right) /L}\right\} _{n\in \mathbb{\mathbb{Z}}}$ is an
orthonormal basis for $L^{2}\left( vt,L_{2}+vt\right) $, for every $t\geq 0$%
, we rewrite (\ref{29}) as 
\begin{equation}
\frac{1}{4\pi i}\sqrt{\frac{2L}{1-v}}\left( \tilde{\phi}_{x}+\tilde{\phi}%
_{t}\right) =\sum_{n\in \mathbb{\mathbb{Z}}^{\ast }}nc_{n}\sqrt{\frac{1-v}{2L%
}}e^{n\pi i\left( 1-v\right) \left( t+x\right) /L},  \label{30}
\end{equation}%
for $x\in \left( vt,L_{2}+vt\right) $. This means that $nc_{n}$ is the $%
n^{th}$ coefficient of the function%
\begin{equation}
\frac{1}{4\pi i}\sqrt{\frac{2L}{1-v}}\left( \tilde{\phi}_{x}+\tilde{\phi}%
_{t}\right) \in L^{2}\left( vt,L_{2}+vt\right) .  \label{parsev+}
\end{equation}%
By consequence,%
\begin{equation}
nc_{n}=\frac{1}{4\pi i}\int_{vt}^{L_{2}+vt}\left( \tilde{\phi}_{x}+\tilde{%
\phi}_{t}\right) e^{-n\pi i\left( 1-v\right) \left( t+x\right) /L}dx\text{,
\ \ \ for\ }n\in \mathbb{\mathbb{Z}}^{\ast }  \label{31-}
\end{equation}%
and (\ref{cn+}) holds as claimed for $t=0$.

The same argument can be carried out on the interval $\left(
-L_{1}+vt,L+vt\right)$ by taking this time the difference between (\ref{phx}%
) and (\ref{pht}), we obtain%
\begin{equation*}
\tilde{\phi}_{x}-\tilde{\phi}_{t}=\left\{ 
\begin{array}{ll}
\frac{2\pi i}{L}\gamma _{v}\left( 1-v\right) \displaystyle\sum_{n\in \mathbb{%
\ \mathbb{Z} }^{\ast }}nc_{n}e^{n\pi i\left( 1-v\right) \left( \left(
1+v\right) t+\gamma _{v}\left( vt-x\right) \right) /L}, & x\in \left(
-L_{1}+vt,vt\right) ,\smallskip \\ 
\frac{2\pi i}{L}\left( 1+v\right) \displaystyle\sum_{n\in \mathbb{\mathbb{Z} 
}^{\ast }}nc_{n}e^{n\pi i\left( 1+v\right) \left( t-x\right) /L}, & x\in
\left( vt,L+vt\right) .%
\end{array}%
\right.
\end{equation*}%
After few rearrangement, it follows that%
\begin{equation}
\tilde{\phi}_{x}-\tilde{\phi}_{t}=\frac{2\pi i}{L}\left( 1+v\right)
\sum_{n\in \mathbb{\mathbb{Z}}^{\ast }}nc_{n}e^{n\pi i\left( 1+v\right)
\left( t-x\right) /L},\text{ \ \ for\ }x\in \left( -L_{1}+vt,L+vt\right) .
\label{29b}
\end{equation}%
Since $\left\{ \sqrt{\frac{1+v}{2L}}e^{n\pi i\left( 1+v\right) \left(
t-x\right) /L}\right\} _{n\in \mathbb{\mathbb{Z}}}$ is an orthonormal basis
for $L^{2}\left( -L_{1}+vt,L+vt\right) $, we deduce that%
\begin{equation}
nc_{n}=\frac{1}{4\pi i}\int_{-L_{1}+vt}^{L+vt}\left( \tilde{\phi}_{x}-\tilde{%
\phi}_{t}\right) e^{-n\pi i\left( 1+v\right) \left( t-x\right) /L}dx\text{,
\ \ \ for\ }n\in \mathbb{\mathbb{Z}}^{\ast }.  \label{31}
\end{equation}%
For $t=0$, we obtain (\ref{cn-}) and the theorem follows.
\end{proof}

As a byproduct of the above proof, we have the following.

\begin{corollary}
\label{coroSll}Under the assumptions (\ref{tlike}) and (\ref{ic}), the sum $%
\sum_{n\in \mathbb{\mathbb{Z}}^{\ast }}\left\vert nc_{n}\right\vert ^{2}$ is
finite and is given by any of the two formulas, for $t\geq 0,$%
\begin{align}
\sum_{n\in \mathbb{\mathbb{Z}}^{\ast }}\left\vert nc_{n}\right\vert ^{2}& =%
\frac{L}{8\pi ^{2}\left( 1-v\right) }\int_{vt}^{L_{2}+vt}\left( \tilde{\phi}%
_{x}+\tilde{\phi}_{t}\right) ^{2}dx  \label{ncn+} \\
& =\frac{L}{8\pi ^{2}\left( 1+v\right) }\int_{-L_{1}+vt}^{L+vt}\left( \tilde{%
\phi}_{x}-\tilde{\phi}_{t}\right) ^{2}dx.  \label{ncn-}
\end{align}
\end{corollary}

\begin{proof}
Parseval's equality applied to the function given in (\ref{30}) yields%
\begin{equation*}
\sum_{n\in \mathbb{\mathbb{Z} }^{\ast }}\left\vert nc_{n}\right\vert
^{2}=\left\vert \frac{1}{4\pi i}\sqrt{\frac{2L}{1-v}}\right\vert
^{2}\int_{vt}^{L_{2}+vt}\left( \tilde{\phi}_{x}+\tilde{\phi}_{t}\right)
^{2}dx\text{, \ \ \ for\ }t\geq 0.
\end{equation*}%
Thus (\ref{ncn+}) holds as claimed. The identity (\ref{ncn-}) follows from (%
\ref{31}) in a similar manner.
\end{proof}

\subsection{A numerical example}

To illustrate the above results, we compute the solution of (\ref{wave}) for
two values of speed $v=0.3$, $v=0.7$ and 
\begin{equation*}
L=\pi \ ,\phi ^{0}\left( x\right) =\sin \left( x\right) /10,\ \ \phi
^{1}\left( x\right) =0
\end{equation*}%
and use (\ref{ncn+}) for the first 40 frequencies, i.e. $\left\vert
n\right\vert \leq 40 $ in the series sum (\ref{exact0}). See Figures 4 and 5.

\begin{figure}[h]
\centering
\begin{minipage}[h]{0.48\textwidth}
    \includegraphics[width=\textwidth]{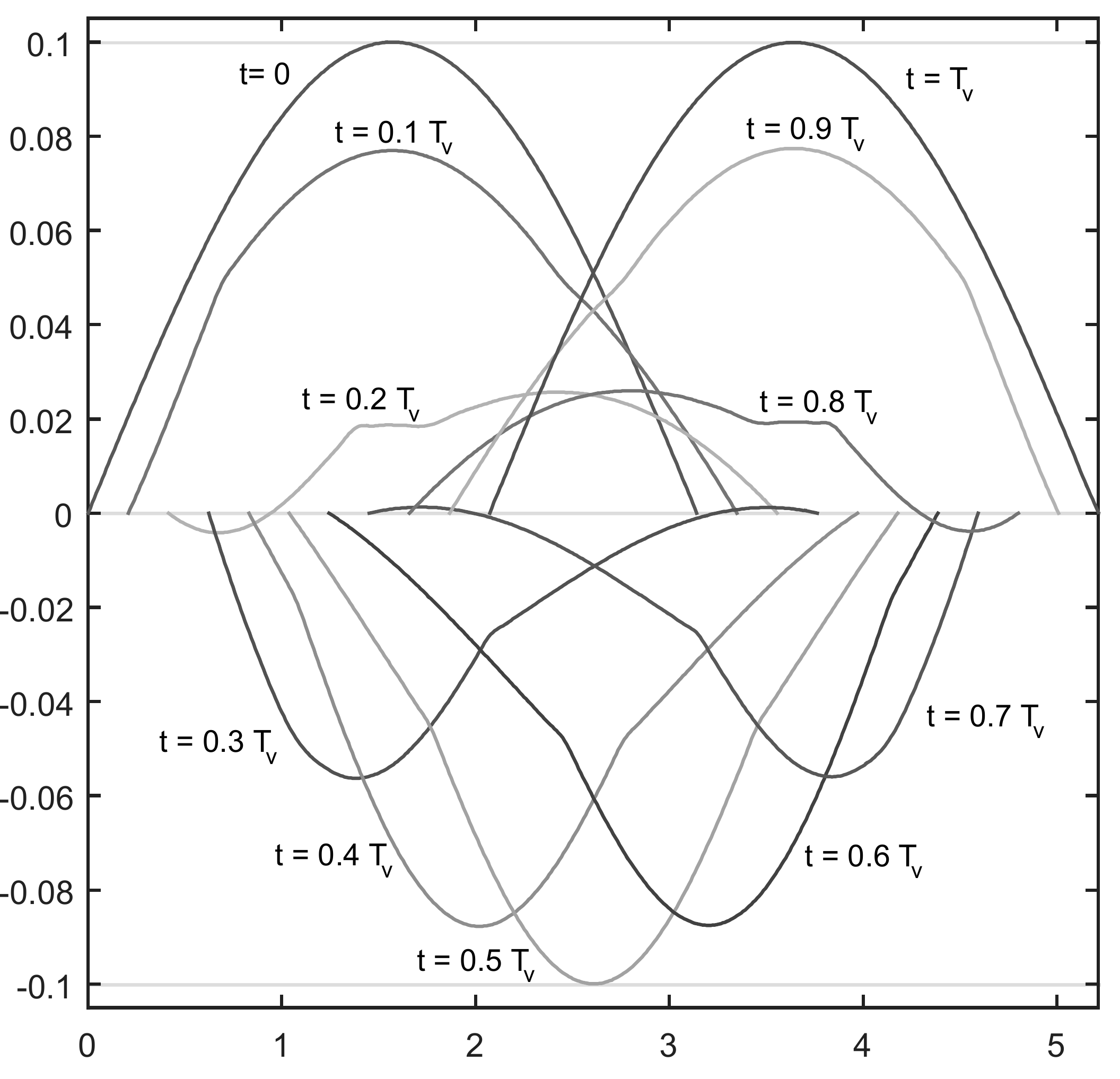}
    \caption{The solution $\protect\phi$ for $v=0.3$ in the interval $(vt,%
\protect\pi+vt)$ over one period $T_v \simeq 6.91$.}
  \end{minipage}
\hfill 
\begin{minipage}[h]{0.48\textwidth}
    \includegraphics[width=\textwidth]{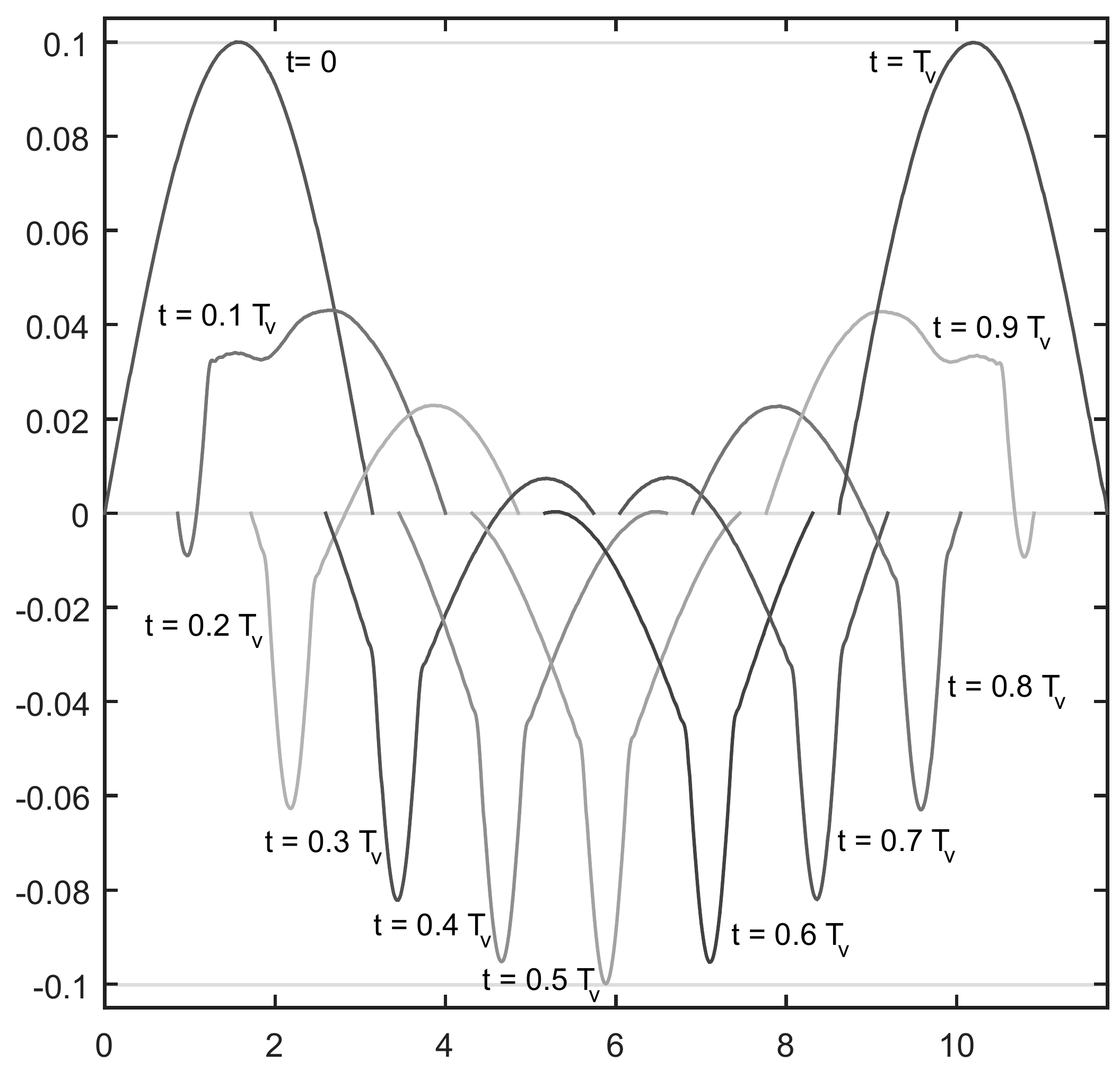}
    \caption{The solution $\protect\phi$ for $v=0.7$ in the interval $(vt,%
\protect\pi+vt)$ over one period $T_v \simeq 12.32$.}
  \end{minipage}
\end{figure}

\section{Energy expressions and estimates}

In this section, we show that the energy $\mathcal{E}_{v}\left( t\right) $
of the solution of Problem (\ref{wave}) is conserved in time.

\begin{theorem}
\label{th1}Under the assumptions (\ref{tlike}) and (\ref{ic}), the solution
of Problem (\ref{wave}) satisfies%
\begin{equation}
\mathcal{E}_{v}\left( t\right) =\frac{2\pi ^{2}\left( 1-v^{2}\right) }{L}%
\sum_{n\in \mathbb{\mathbb{Z}}^{\ast }}\left\vert nc_{n}\right\vert ^{2},\ \
\ \ \ \text{for}\ t\geq 0.  \label{est0}
\end{equation}%
$($the left-hand side is independent of $t).$
\end{theorem}

\begin{proof}
The two identities (\ref{ncn+}) and (\ref{ncn-}) implies that%
\begin{equation}
\frac{1}{1+v}\int_{-L_{1}+vt}^{L+vt}\left( \tilde{\phi}_{x}-\tilde{\phi}%
_{t}\right) ^{2}dx=\frac{1}{1-v}\int_{vt}^{L_{2}+vt}\left( \tilde{\phi}_{x}+%
\tilde{\phi}_{t}\right) ^{2}dx=\frac{8\pi ^{2}}{L}\sum_{n\in \mathbb{\ 
\mathbb{Z}}^{\ast }}\left\vert nc_{n}\right\vert ^{2}.  \label{32}
\end{equation}%
Using the extensions (\ref{phi0x+}), (\ref{phi1+}) and considering the
change of variable $x=\frac{1}{\gamma _{v}}\left( vt-\xi \right) +vt$, in $%
\left( -L_{1}+vt,vt\right) ,$ we obtain%
\begin{multline*}
\frac{1}{1+v}\int_{-L_{1}+vt}^{vt}\left( \tilde{\phi}_{x}\left( x,t\right) -%
\tilde{\phi}_{t}\left( x,t\right) \right) ^{2}dx=-\frac{1}{1+v}%
\int_{L+vt}^{vt}\gamma _{v}\left( \phi _{x}\left( \xi ,t\right) +\phi
_{t}\left( \xi ,t\right) \right) ^{2}d\xi \\
=\frac{1}{1-v}\int_{vt}^{L+vt}\left( \phi _{x}\left( \xi ,t\right) +\phi
_{t}\left( \xi ,t\right) \right) ^{2}d\xi .
\end{multline*}%
Taking $x=\gamma _{v}\left( vt-\xi \right) +\frac{2L}{v-1}+vt$, in $\left(
L+vt,L_{2}+vt\right) ,$ we obtain%
\begin{equation*}
\frac{1}{1-v}\int_{L+vt}^{L_{2}+vt}\left( \tilde{\phi}_{x}\left( x,t\right) +%
\tilde{\phi}_{t}\left( x,t\right) \right) ^{2}dx=\frac{1}{1+v}%
\int_{vt}^{L+vt}\left( \phi _{x}\left( \xi ,t\right) -\phi _{t}\left( \xi
,t\right) \right) ^{2}d\xi .
\end{equation*}%
Then, taking (\ref{32}) into account,\ it comes that 
\begin{multline*}
\frac{1}{1-v}\int_{-L_{1}+vt}^{L+vt}\left( \tilde{\phi}_{t}+\tilde{\phi}%
_{x}\right) ^{2}dx+\frac{1}{1+v}\int_{vt}^{L_{2}+vt}\left( \tilde{\phi}_{x}-%
\tilde{\phi}_{t}\right) ^{2}dx \\
=\frac{2}{1+v}\int_{vt}^{L+vt}\left( \phi _{x}-\phi _{t}\right) ^{2}dx+\frac{%
2}{1-v}\int_{vt}^{L+vt}\left( \phi _{t}+\phi _{x}\right) ^{2}dx=\frac{16\pi
^{2}}{L}\sum_{n\in \mathbb{\mathbb{Z}}^{\ast }}\left\vert nc_{n}\right\vert
^{2}.
\end{multline*}%
Expanding $\left( \phi _{x}\pm \phi _{t}\right) ^{2}$ and collecting similar
terms, we get 
\begin{equation}
\frac{1}{1-v^{2}}\left( 2\int_{vt}^{L+vt}\phi _{x}^{2}+\phi _{t}^{2}+4v\phi
_{x}\phi _{t}dx\right) =\frac{8\pi ^{2}}{L}\sum_{n\in \mathbb{\mathbb{Z}}%
^{\ast }}\left\vert nc_{n}\right\vert ^{2},\text{ \ \ for}\ t\geq 0.
\label{est1}
\end{equation}%
Recalling that $\mathcal{E}_{v}\left( t\right) $ is given by (\ref{E}), this
identity can be rewritten as in (\ref{est0}). This end the proof.
\end{proof}

The fact that $\mathcal{E}_{v}\left( t\right) $ is constant in time can be
established by using only the identities $\phi _{tt}=\phi _{xx}$ and $\phi
\left( vt,t\right) =\phi \left( L+vt,t\right) =0$ from (\ref{wave}).

\begin{proof}[A second proof for the conservation of $\mathcal{E}_{v}\left(
t\right) $]
It suffices to show that $\frac{d}{dt}\mathcal{E}_{v}\left( t\right) =0.$
First, the boundary conditions $\phi \left( vt,t\right) =\phi \left(
L+vt,t\right) =0$ means that $\frac{d}{dt}\phi \left( vt,t\right) =\frac{d}{%
dt}\phi \left( L+vt,t\right) =0,$ hence%
\begin{equation}
\phi _{t}\left( vt,t\right) +v\phi _{x}\left( vt,t\right) =\phi _{t}\left(
L+vt,t\right) +v\phi _{x}\left( L+vt,t\right) =0.  \label{dtotal=0}
\end{equation}%
Since the limits of the integral in the expression of $\mathcal{E}_{v}\left(
t\right) $ are time-dependent, then Leibnitz's rule implies that%
\begin{multline}
\frac{d}{dt}\mathcal{E}_{v}\left( t\right) =v\left( 1-v^{2}\right) \left(
\phi _{x}^{2}\left( L+vt,t\right) -\phi _{x}^{2}\left( vt,t\right) \right)
\label{DE} \\
+\int_{vt}^{L+vt}\frac{\partial }{\partial t}\left( \phi _{t}+v\phi
_{x}\right) ^{2}\ dx+\left( 1-v^{2}\right) \frac{\partial }{\partial t}%
\left( \phi _{x}^{2}\right) dx.
\end{multline}%
The remaining integral equals, after using $\phi _{tt}=\phi _{xx}\ $then
integrating by parts, 
\begin{multline*}
\int_{vt}^{L+vt}\left( \phi _{t}+v\phi _{x}\right) \phi _{xx}+\left( v\phi
_{t}+\phi _{x}\right) \phi _{xt}dx=\int_{vt}^{L+vt}-\left( \phi _{xt}+v\phi
_{xx}\right) \phi _{x}+\left( v\phi _{t}+\phi _{x}\right) \phi _{xt}dx \\
=v\int_{vt}^{L+vt}-\phi _{xx}\phi _{x}+\phi _{t}\phi _{xt}dx,
\end{multline*}%
which is nothing but%
\begin{equation*}
v\int_{vt}^{L+vt}\frac{\partial }{\partial x}\left( \phi _{t}^{2}-\phi
_{x}^{2}\right) dx=-v\left( 1-v^{2}\right) \left( \phi _{x}^{2}\left(
L+vt,t\right) -\phi _{x}^{2}\left( vt,t\right) \right)
\end{equation*}%
due to (\ref{dtotal=0}). Going back to (\ref{DE}), we infer that $\frac{d}{dt%
}\mathcal{E}_{v}\left( t\right) =0$ as claimed.
\end{proof}

Let us now compare $\mathcal{E}_{v}\left( t\right) $ to the usual expression
of energy for the wave equation 
\begin{equation*}
E_{v}\left( t\right) :=\frac{1}{2}\int_{vt}^{L+vt}\phi _{t}^{2}+\phi
_{x}^{2}dx,\text{ \ for}\ t\geq 0.
\end{equation*}%
In contrast with $\mathcal{E}_{v}\left( t\right) ,$ the expression $%
E_{v}\left( t\right) $ is not conserved in general. Due to the periodicity
relation (\ref{T-period}), we know at least that $E_{v}$ is $T_{v}-$periodic
in time. Moreover we have

\begin{corollary}
\label{coro3.1}Under the assumptions (\ref{tlike}) and (\ref{ic})\emph{,}
the energy $E_{v}\left( t\right) $ of the solution of Problem (\ref{wave})
satisfies%
\begin{equation}
\frac{\mathcal{E}_{v}\left( t\right) }{1+v}\leq E_{v}\left( t\right) \leq 
\frac{\mathcal{E}_{v}\left( t\right) }{1-v},\ \ \ \ \ \text{for}\ t\geq 0
\label{ES0}
\end{equation}%
and%
\begin{equation}
\frac{1}{\gamma _{v}}E_{v}\left( 0\right) \leq E_{v}\left( t\right) \leq
\gamma _{v}E_{v}\left( 0\right) ,\ \ \ \ \ \text{\ for }t\geq 0.
\label{stab}
\end{equation}
\end{corollary}

\begin{proof}
We can write (\ref{est1}) as 
\begin{equation}
E_{v}\left( t\right) +v\int_{vt}^{L+vt}\phi _{x}\phi _{t}\ dx=\mathcal{E}%
_{v}\left( t\right) ,\text{ \ \ for}\ t\geq 0.  \label{En3.8}
\end{equation}%
Thanks to the algebraic inequality $\pm ab\leq \left( a^{2}+b^{2}\right) /2$
we know that%
\begin{equation*}
\pm \int_{vt}^{L+vt}\phi _{x}\phi _{t}\ dx\leq E_{v}\left( t\right) \text{,
\ \ \ for }t\geq 0.
\end{equation*}%
Then, it comes that%
\begin{equation}
\mathcal{E}_{v}\left( t\right) \leq \left( 1+v\right) E_{v}\left( t\right) 
\text{ \ \ and \ \ }\left( 1-v\right) E_{v}\left( t\right) \leq \mathcal{E}%
_{v}\left( t\right) ,\text{ \ for }t\geq 0.  \label{ESES}
\end{equation}%
This implies (\ref{ES0}). Since (\ref{ESES}) holds also for $t=0$, then (\ref%
{stab}) follows by combining the two inequalities%
\begin{align*}
\left( 1-v\right) E_{v}\left( t\right) & \leq \mathcal{E}_{v}\left( t\right)
=\mathcal{E}_{v}\left( 0\right) \leq \left( 1+v\right) E_{v}\left( 0\right) ,
\\
\left( 1-v\right) E_{v}\left( 0\right) & \leq \mathcal{E}_{v}\left( 0\right)
=\mathcal{E}_{v}\left( t\right) \leq \left( 1+v\right) E_{v}\left( t\right) ,
\end{align*}%
for\ $t\geq 0$.
\end{proof}

\begin{remark}
The equality in estimation (\ref{ES0}) may hold for some $t\geq 0$. This is
the case whenever $\phi_{t}\left( x,t\right)=\pm \phi _{x}\left(x, t\right)$%
, for $x\in \mathbf{I}_t$ and some $t\geq 0$. For instance, if the initial
data satisfy $\phi ^{1}=\pm \phi _{x}^{0}$ we obtain from (\ref{En3.8}) that 
\begin{equation*}
\left( 1\pm v\right) E_{v}\left( 0\right) =E_{v}\left( 0\right)
+v\int_{0}^{L}\phi _{x}^{0}\phi ^{1}\ dx=\mathcal{E}_{v}\left( 0\right) ,
\end{equation*}%
i.e. $E_{v}\left( 0\right) =\mathcal{E}_{v}\left( 0\right) /\left( 1\pm
v\right) $. By periodicity, we have also $E_{v}\left( nT_{v}\right) =%
\mathcal{E}_{v}\left( 0\right) /\left( 1\pm v\right) ,$ for $n\in \mathbb{Z} 
$. The + and -- signs are used respectively.
\end{remark}

\begin{remark}
As $v\rightarrow 1^{-},$ we have $\mathcal{E}_{v}\left( 0\right) \rightarrow
\left\Vert \phi ^{1}+\phi _{x}^{0}\right\Vert _{L^{2}\left( 0,L\right) }/2.$
If the initial data satisfies $\phi ^{1}+\phi _{x}^{0}\neq 0,$ it follows
from (\ref{ES0}) that%
\begin{equation*}
E_{v}\left( t\right) \leq \frac{\mathcal{E}_{v}\left( t\right) }{1-v}=\frac{%
\mathcal{E}_{v}\left( 0\right) }{1-v}\rightarrow +\infty \text{, \ as }%
v\rightarrow 1^{-}.
\end{equation*}
Taking the precedent remark into account, we may have large value for $%
E_{v}\left( t\right)$, as $v$ becomes close to the speed of propagation $c=1$%
, even for small initial value $\mathcal{E}_{v}\left( 0\right)$. To see what
happens to the string in this case, let us take $v=0.9$ in the precedent
numerical example, see Figure \ref{fig09}. We observe a layer effect (i.e. a
subregion in $\mathbf{I}_t$ where $\phi _{x}$ becomes very large) that
travels from the left endpoint to the right one over one period $T_{v}$.
This phenomenon becomes more marked as $v$ is closer to $1$.
\end{remark}

\begin{figure}[tbph]
\centering
\includegraphics[width=0.57\textwidth,height=77mm]{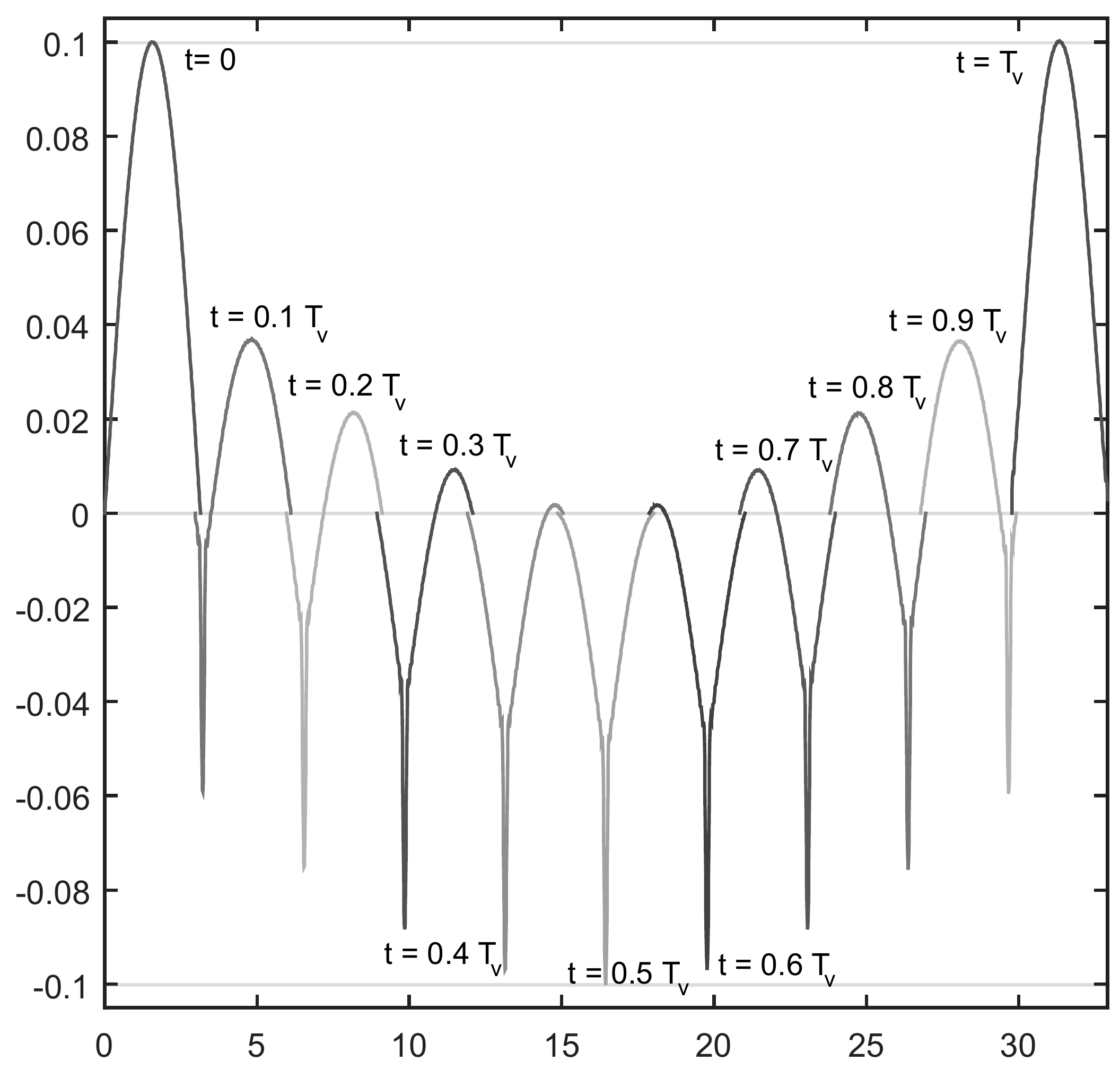}
\caption{The solution $\protect\phi$ for $v=0.9$ in the interval $(vt,%
\protect\pi+vt)$ over one period $T_v \simeq 33.07$.}
\label{fig09}
\end{figure}


\section{Boundary observability}

In many applications, it is preferred that the sensors do not interfere with
the vibrations of the string, so they are placed at the extremities. In
addition, interior pointwise sensors are difficult to design and the system
may become unobservable depending on the sensors location. This fact was
shown by Yang and Mote in \cite{YaMo1991} where they cast $\left( \ref{freez}%
\right) $ in a state space form and use semi-group theory.

\subsection{Observability at one endpoint}

First, we show the observability of (\ref{wave}) at each endpoint $x_{b}+vt$
where 
\begin{equation*}
x_{b}=0\ \ \text{or \ }x_{b}=L.
\end{equation*}%
The problem of observability considered here can be stated as follows: To
give sufficient conditions on the length $T$ of the time interval such that
there exists a constant $C(T)>0$ for which the observability inequality%
\footnote{%
One can replace $\mathcal{E}_{v}\left( 0\right) $ by $E_{v}\left( 0\right) $
in the left-hand side, but this does not matter since (\ref{ES0}) holds
under the assumption (\ref{tlike}).}%
\begin{equation}
\mathcal{E}_{v}\left( 0\right) \leq C(T)\int_{0}^{T}\phi
_{x}^{2}(x_{b}+vt,t)dt,  \label{ct}
\end{equation}%
holds for all the solutions of (\ref{wave}). This inequality is also called
the inverse inequality.

The next theorem shows in particular that the boundary observability holds
for $T\geq T_{v}=2L/\left( 1-v^{2}\right) $.

\begin{theorem}
\label{thobs1}Under the assumptions (\ref{tlike}) and (\ref{ic}), we have:%
\begin{equation}
\int_{0}^{MT_{v}}\phi _{x}^{2}(x_{b}+vt,t)dt=\frac{4M}{\left( 1-v^{2}\right)
^{2}}\mathcal{E}_{v}\left( 0\right) .  \label{33}
\end{equation}

By consequence, the solution of (\ref{wave}) satisfies the direct inequality%
\begin{equation}
\int_{0}^{T}\phi _{x}^{2}(x_{b}+vt,t)dt\leq K_{1}(v,T)\mathcal{E}_{v}\left(
0\right) \text{, for every }T\geq 0,  \label{D1}
\end{equation}%
with a constant $K_{1}(v,T)$ depending only on $v$ and $T$.

If $T\geq T_{v}$, Problem (\ref{wave}) is observable at $\xi \left( t\right)
=x_{b}+vt$ and it holds that:%
\begin{equation}
\mathcal{E}_{v}\left( 0\right) \leq \frac{\left( 1-v^{2}\right) ^{2}}{4}%
\int_{0}^{T}\phi _{x}^{2}(x_{b}+vt,t)dt.  \label{obs1}
\end{equation}
\end{theorem}

\begin{proof}
Thanks to (\ref{ph x}), we can evaluate $\phi _{x}$ at the endpoint $%
x=x_{b}+vt$ . We obtain%
\begin{align*}
\phi _{x}(x_{b}+vt,t)& =\frac{\pi i}{L}\sum_{n\in \mathbb{\mathbb{Z}}^{\ast
}}nc_{n}\left( \left( 1-v\right) e^{\frac{n\pi i\left( 1-v\right) }{L}\left(
\left( 1+v\right) t+x_{b}\right) }+\left( 1+v\right) e^{\frac{n\pi i\left(
1+v\right) }{L}\left( \left( 1-v\right) t-x_{b}\right) }\right) \\
& =\frac{\pi i}{L}\sum_{n\in \mathbb{\mathbb{Z}}^{\ast }}nc_{n}\left( \left(
1-v\right) e^{\frac{n\pi i\left( 1-v\right) }{L}x_{b}}+\left( 1+v\right) e^{-%
\frac{n\pi i\left( 1+v\right) }{L}x_{b}}\right) e^{n\pi i\left(
1-v^{2}\right) t/L},
\end{align*}%
which can be rewritten as%
\begin{equation}
\phi _{x}(x_{b}+vt,t)=\left\{ 
\begin{array}{ll}
\displaystyle\frac{2\pi i}{L}\sum_{n\in \mathbb{\mathbb{Z}}^{\ast
}}nc_{n}e^{2n\pi it/T_{v}}\smallskip , & \text{if }x_{b}=0, \\ 
\displaystyle\frac{2\pi i}{L}\sum_{n\in \mathbb{\mathbb{Z}}^{\ast
}}nc_{n}e^{-n\pi i\left( 1+v\right) }e^{2n\pi it/T_{v}}, & \text{if }x_{b}=L.%
\end{array}%
\right. .  \label{phix0}
\end{equation}%
Let $M\in \mathbb{N}^{\ast }$. Since the set of functions $\left\{ e^{2n\pi
it/T_{v}}/\sqrt{T_{v}}\right\} _{n\in \mathbb{\mathbb{Z}}}$ is complete and
orthonormal in the space $L^{2}(mT_{v},\left( m+1\right) T_{v})$ for $%
m=0,...,M-1$, then Parseval's equality applied to the functions%
\begin{equation*}
\phi _{x}(x_{b}+vt,t)\in L^{2}(mT_{v},\left( m+1\right) T_{v}),\text{ \ for }%
m=0,...,M-1,
\end{equation*}%
yields, after summing up the integrals for all the subintervals of $\left(
0,MT_{v}\right) ,$%
\begin{equation*}
\frac{1}{T_{v}}\int_{0}^{MT_{v}}\phi _{x}^{2}(x_{b}+vt,t)dt=\frac{4M\pi ^{2}%
}{L^{2}}\sum_{n\in \mathbb{\mathbb{Z}}^{\ast }}\left\vert nc_{n}\right\vert
^{2}
\end{equation*}%
and (\ref{33}) follows.

For every $T\geq 0$, we can take the integer $M$ large enough to satisfy $%
MT_{v}=M\frac{2L}{1-v^{2}}\geq T$. Then, the identity (\ref{33}) yields 
\begin{equation*}
\int_{0}^{T}\phi _{x}^{2}(x_{b}+vt,t)dt\leq \int_{0}^{MT_{v}}\phi
_{x}^{2}(x_{b}+vt,t)dt=\frac{4M}{\left( 1-v^{2}\right)^2}\mathcal{E}%
_{v}\left( 0\right),
\end{equation*}%
i.e., (\ref{D1}) holds for $K_{1}(v,T):=4M/\left(1-v^{2}\right)^2$. The
inequality (\ref{obs1}) follows from (\ref{33}) with $M=1.$
\end{proof}

\begin{remark}
\label{rmk dt}Taking (\ref{dtotal=0}) into account, we have%
\begin{equation*}
\phi _{t}^{2}(x_{b}+vt,t)=v^{2}\phi _{x}^{2}(x_{b}+vt,t),\text{ \ \ \ for }%
x_{b}=0\ \text{or }x_{b}=L,\ \forall t\geq 0.
\end{equation*}
Then, the results of Theorem \ref{thobs1} hold if we replace $\phi
_{x}(x_{b}+vt,t)$ by $\phi _{t}(x_{b}+vt,t)/v^{2}$ with the same constants
in the inequalities.
\end{remark}

\begin{remark}
The time of boundary observability $T_{v}$ can be predicted by a simple
argument, see Figure $\ref{fig4}$. An initial disturbances concentrated near 
$x=L+vt$ may propagate to the left as $t$ increases. It reaches the left
boundary, when $t$ is close to $\frac{L}{1+v}$. Then travels back to reach
the right boundary when $t$ is close to $\frac{2L}{1-v^{2}}=T_{v}$, see
Figure \ref{fig4} (left). We need the same time $T_{v}$ for an initial
disturbance concentrated near $x=vt$, see Figure \ref{fig4} (right).
\end{remark}

\begin{figure}[tbph]
\centering\includegraphics[width=0.71\textwidth,height=57mm]{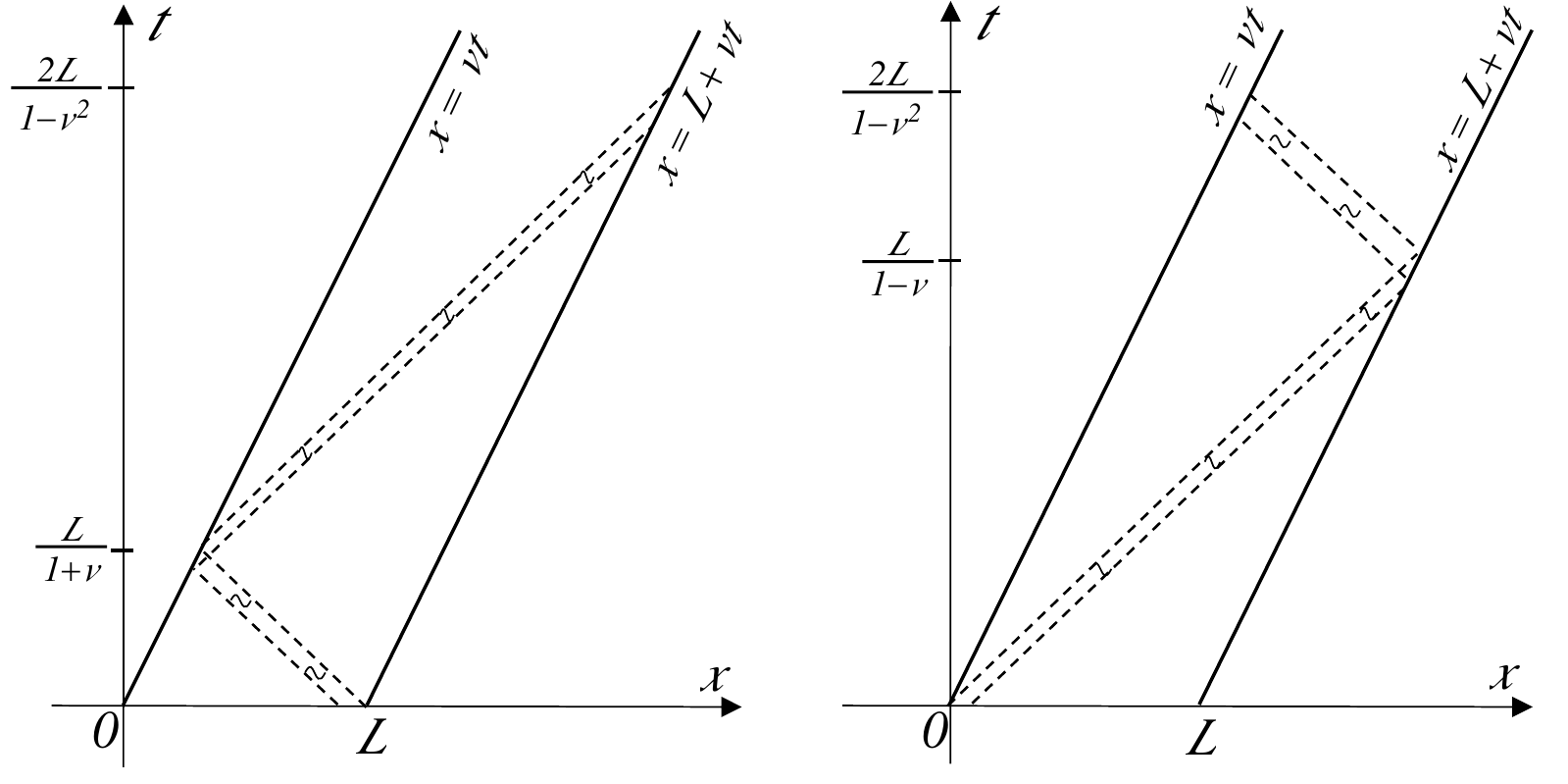}
\caption{Propagation of small disturbances with support near an endpoint.}
\label{fig4}
\end{figure}

\subsection{Observability at both endpoints}

Place two sensors at both endpoints $x=vt$ and $x=L+vt$ of the interval $%
\mathbf{I}_{t}$, one expects a shorter time of observability. The next
theorem shows that the observability, in this case, holds for $T\geq \tilde{T%
}_{v}=L/\left( 1-v\right) $.

\begin{theorem}
\label{thobs2}Under the assumption (\ref{tlike}) and (\ref{ic}),\emph{\ }we
have: 
\begin{equation}
\int_{0}^{\frac{L}{1+v}}\phi _{x}^{2}(vt,t)dt+\int_{0}^{\frac{L}{1-v}}\phi
_{x}^{2}(L+vt,t)dt=\frac{4\mathcal{E}_{v}\left( 0\right)}{\left(
1-v^{2}\right)^{2}} .  \label{slm2}
\end{equation}%
By consequence, the solution of (\ref{wave}) satisfies the direct inequality%
\begin{equation}
\int_{0}^{T}\phi _{x}^{2}(vt,t)+\phi _{x}^{2}(L+vt,t)dt\leq K_{2}(v,T)%
\mathcal{E}_{v}\left( 0\right) \text{, for every }T\geq 0,  \label{D2}
\end{equation}%
with a constant $K_{2}(v,T)$ depending only on $v$ and $T$.

If $T\geq \tilde{T}_{v}$, Problem (\ref{wave}) is observable at both
endpoints $x=vt,x=L+vt$ and it holds that%
\begin{equation}
\mathcal{E}_{v}\left( 0\right) \leq \frac{\left( 1-v^{2}\right) ^{2}}{4}%
\int_{0}^{T}\phi _{x}^{2}(vt,t)+\phi _{x}^{2}(L+vt,t)dt.  \label{obs3}
\end{equation}
\end{theorem}

\begin{proof}
Arguing by density as in \cite{Seng2020}, it suffices to establish (\ref%
{slm2}) for smooth initial data. Thus, assuming that $\phi _{x}^{0}$ and $%
\phi ^{1}$ are continuous functions ensures in particular that their Fourier
series are absolutely converging. This allow us to to interchange summation
and integration in the infinite series considered in the remainder of the
proof.

Let $m\in \mathbb{Z}^{\ast }$. On one hand, taking $x_{b}=0$ in (\ref{phix0}%
), multiplying by $\overline{imc_{m}e^{2m\pi it/T_{v}}}$ then integrating on 
$\left( 0,L/\left( 1+v\right) \right) ,$ we obtain 
\begin{equation*}
\int_{0}^{\frac{L}{1+v}}\phi _{x}(vt,t)\text{ }\overline{imc_{m}e^{2m\pi
it/T_{v}}}dt=\frac{2\pi }{L}m\bar{c}_{m}\int_{0}^{\frac{L}{1+v}}\left(
\sum_{n\in \mathbb{\mathbb{Z}}^{\ast }}nc_{n}e^{2\left( n-m\right) \pi
it/T_{v}}\right) dt.
\end{equation*}%
Integrating term-by-term, we obtain%
\begin{multline}
\int_{0}^{\frac{L}{1+v}}\phi _{x}(vt,t)\text{ }\overline{imc_{m}e^{2m\pi
it/T_{v}}}dt=\frac{2\pi }{L}\sum_{n\in \mathbb{\mathbb{Z}}^{\ast }}nmc_{n}%
\bar{c}_{m}\int_{0}^{\frac{L}{1+v}}e^{2\left( n-m\right) \pi it/T_{v}}dt
\label{Anm} \\
=\sum_{n\in \mathbb{\mathbb{Z}}^{\ast }}A_{nm},
\end{multline}%
where%
\begin{equation*}
A_{nm}=\left\{ 
\begin{array}{ll}
\displaystyle\frac{2\pi }{1+v}\left\vert mc_{m}\right\vert ^{2},\medskip  & 
\text{ \ if }n=m, \\ 
\displaystyle\frac{2nmc_{n}\bar{c}_{m}}{i\left( n-m\right) \left(
1-v^{2}\right) }\left( e^{\pi i\left( n-m\right) \left( 1-v\right)
}-1\right) , & \text{ \ if }n\neq m.%
\end{array}%
\right. 
\end{equation*}

On the other hand, taking $x_{b}=L$ in the identity (\ref{phix0}),
multiplying by $\overline{imc_{m}e^{-m\pi i\left( 1+v\right) }e^{2m\pi
it/T_{v}}}$, then integrating term-by-term on $\left( 0,L/\left( 1-v\right)
\right) $, we end up with%
\begin{equation}
\int_{0}^{\frac{L}{1-v}}\phi _{x}(L+vt,t)\text{ }\overline{imc_{m}e^{-m\pi
i\left( 1+v\right) }e^{2m\pi it/T_{v}}}dt=\sum_{n\in \mathbb{\mathbb{Z}}%
^{\ast }}B_{nm},  \label{Bnm}
\end{equation}%
where%
\begin{equation*}
B_{nm}=\left\{ 
\begin{array}{ll}
\displaystyle\frac{2\pi }{1-v}\left\vert mc_{m}\right\vert ^{2},\medskip & 
\text{ \ if }n=m, \\ 
\displaystyle\frac{2nmc_{n}\bar{c}_{m}}{i\left( n-m\right) \left(
1-v^{2}\right) }\left( 1-e^{-\left( n-m\right) \pi i\left( 1+v\right)
}\right) , & \text{ \ if }n\neq m.%
\end{array}%
\right.
\end{equation*}%
Computing $A_{nm}+B_{nm}$ we obtain:

\begin{itemize}
\item If $n=m,$ then 
\begin{equation*}
A_{mm}+B_{mm}=2\pi \left\vert mc_{m}\right\vert ^{2}\left( \frac{1}{1+v}+%
\frac{1}{1-v}\right) =\frac{4\pi }{1-v^{2}}\left\vert mc_{m}\right\vert ^{2}.
\end{equation*}

\item If $n\neq m,$ then%
\begin{align*}
A_{nm}+B_{nm}& =\frac{2nmc_{n}\bar{c}_{m}}{i\left( n-m\right) \left(
1-v^{2}\right) }\left( e^{\pi i\left( n-m\right) \left( 1-v\right)
}-e^{-\left( n-m\right) \pi i\left( 1+v\right) }\right) \\
& =\frac{2nmc_{n}\bar{c}_{m}}{i\left( n-m\right) \left( 1-v^{2}\right) }%
e^{-\pi i\left( n-m\right) \left( 1-v\right) }\left( e^{\left( n-m\right)
\pi i\left( 1-v+1+v\right) }-1\right) ,
\end{align*}

i.e., $A_{nm}+B_{nm}=0\text{ if }n\neq m.$
\end{itemize}

By consequence, the sum of (\ref{Anm}) and (\ref{Bnm}) is simply given by%
\begin{multline}
\int_{0}^{\frac{L}{1+v}}\phi _{x}(vt,t)\overline{imc_{m}e^{2m\pi it/T_{v}}}dt
\label{A+B} \\
+\int_{0}^{\frac{L}{1-v}}\phi _{x}(L+vt,t)\text{ }\overline{imc_{m}e^{-m\pi
i\left( 1+v\right) }e^{2m\pi it/T_{v}}}dt=\frac{4\pi }{1-v^{2}}\left\vert
mc_{m}\right\vert ^{2},
\end{multline}%
for every $m\in \mathbb{Z}^{\ast }$. Taking the sum for $m\in \mathbb{\ 
\mathbb{Z}}^{\ast }$, and interchange summation and integration, it comes
that%
\begin{multline*}
\int_{0}^{\frac{L}{1+v}}\phi _{x}(vt,t)\left( \sum\limits_{m=-\infty
}^{+\infty }\overline{imc_{m}e^{2m\pi it/T_{v}}}\right) dt \\
+\int_{0}^{\frac{L}{1-v}}\phi _{x}(L+vt,t)\left( \sum\limits_{m=-\infty
}^{+\infty }\overline{imc_{m}e^{-m\pi i\left( 1+v\right) }e^{2m\pi it/T_{v}}}%
\right) dt=\frac{4\pi }{1-v^{2}}\sum\limits_{m=-\infty }^{+\infty
}\left\vert mc_{m}\right\vert ^{2}.
\end{multline*}%
Thanks to (\ref{phix0}), we obtain%
\begin{equation*}
\frac{L}{2\pi }\left( \int_{0}^{\frac{L}{1+v}}\phi
_{x}^{2}(vt,t)dt+\int_{0}^{\frac{L}{1-v}}\phi _{x}^{2}(L+vt,t)dt\right) =%
\frac{4\pi }{1-v^{2}}\sum\limits_{m=-\infty }^{+\infty }\left\vert
mc_{m}\right\vert ^{2}.
\end{equation*}%
This shows (\ref{slm2}).

Inequality (\ref{D2}) is a consequence of Theorem \ref{thobs1}, it suffices
to choose $x_{b}=vt$ then $x_{b}=L+vt$ in the direct inequality (\ref{D1})
and take the sum. The inequality (\ref{obs3}) holds for $T=\max \left\{ 
\frac{L}{1-v},\frac{L}{1+v}\right\} =\tilde{T}_{v}$ and therefore for every $%
T\geq \tilde{T}_{v}$ as well.
\end{proof}

\begin{remark}
If $T<\frac{L}{1-v}$, then the observability does not hold. Indeed, an
initial disturbance with sufficiently small support and close to $x=0$ will
hit the boundary $x=L+vt$ only after the time $T$, see Figure $\ref{fig4}$
(Right).
\end{remark}

\begin{remark}
Thanks to the Hilbert uniqueness method (HUM), due to J-.L. Lions \cite%
{Lion1988}, we can easily derive exact boundary controllability results at
one or at both endpoints from the above observability results. The proof is
not much different from that in \cite{Seng2020}.
\end{remark}

\begin{remark}
The techniques used in this paper can be adapted to deal with more
complicated boundary conditions for travelling strings. The results will
appear in a forthcoming paper.
\end{remark}

\subsection*{Acknowledgements}

The authors have been supported by the General Direction of Scientific
Research and Technological Development (Algerian Ministry of Higher
Education and Scientific Research) PRFU \# C00L03UN280120220010. They are
very grateful to this institution.

\subsection*{ORCID}

Abdelmouhcene Sengouga \href{https://orcid.org/0000-0003-3183-7973}{%
https://orcid.org/0000-0003-3183-7973}

\end{document}